\documentclass[10pt]{amsart} \usepackage[utf8]{inputenc}
\usepackage[english]{babel} \usepackage{tikz-cd} \usepackage{graphicx}
\usepackage{amsthm} \usepackage{amstext}
\usepackage{amsmath,amscd,amsfonts} \usepackage{amssymb}
\usepackage{mathrsfs} \usepackage{mathtools} \usepackage[all]{xy}
\usepackage{stmaryrd} \usepackage{color,comment} \usepackage{amsmath,
  amsthm, amsfonts, enumerate} \usepackage{amsmath}
\usepackage{graphicx} \usepackage[colorlinks=true,
allcolors=blue]{hyperref}

\newtheorem{theorem}{Theorem}[section]

\newtheorem{corollary}[theorem]{Corollary}
\newtheorem{proposition}[theorem]{Proposition}
\newtheorem{definition}[theorem]{Definition} \theoremstyle{remark}
\newtheorem{remark}[theorem]{\bf Remark} \theoremstyle{definition}
\newtheorem{example}[theorem]{\bf Example}
\newtheorem{conjecture}[theorem]{Conjecture}
\numberwithin{equation}{section}

\DeclareMathOperator{\ann}{Ann} \DeclareMathOperator{\inn}{in}
\DeclareMathOperator{\HF}{HF} 

\author{Mats Boij} 
\address{Mats Boij, Department of Mathematics, KTH -- Royal Institute
  of Technology, SE-100 44 Stockholm, Sweden} \email{boij@kth.se}

\author{Luís Duarte}
\address{Luís Duarte, Department of Mathematics, Stockholm University,
  SE-106 91 Stockholm, Sweden} \email{luis.duarte@math.su.se}

\author{Samuel Lundqvist}
\address{Samuel Lundqvist, Department of Mathematics, Stockholm
  University, SE-106 91 Stockholm, Sweden} \email{samuel@math.su.se}

\title[On the initial ideal of a generic artinian Gorenstein
algebra]{On the initial ideal of a generic artinian Gorenstein
  algebra}

\begin{document}

\begin{abstract}
In this note we show that the initial ideal of the annihilator ideal
of a generic form is generated by the largest possible monomials in
each degree.
We also show that the initial ideal with respect to the degree reverse
lexicographical ordering of the annihilator ideal of the complete
symmetric form has this property, by determining a minimal Gr\"obner
basis of it.
Moreover, we determine the total Betti numbers for a class of strongly
stable monomial ideals and show that these numbers agree with those
for the degree reverse lexicographical initial ideals of the ideal
generated by a sufficiently large number of generic forms, and of the
annihilator ideal of a generic form. 
\end{abstract}

\maketitle

\section{Introduction}
Let $R=k[x_1,\ldots,x_n]$ be a standard graded polynomial ring. degree
reverse lexicographical ordering.  Moreno-Soc\'ias
\cite{MORENOSOCIAS2003263} conjectured that the initial ideal $I$ of
an ideal generated by $m$ generic forms is \emph{weakly reverse
  lexicographical}, which means that if $x^{\mu}\in I$ is a minimal
monomial generator of $I$, then every monomial of the same degree as
$x^{\mu}$ that preceeds $x^{\mu}$ with respect to the degree reverse
lexicographical ordering is also in $I$. Pardue \cite{PARDUE2010579}
showed that Moreno-Soc\'ias' conjecture for the special case $m = n$
implies the longstanding Fr\"oberg conjecture \cite{Froberg} on the
Hilbert series of ideals generated by any number of generic forms. We
also recall that a monomial ideal $I$ is
$\emph{reverse lexicographic}$ if, for every $d\ge 0$, $I_d$ is
generated by the largest possible monomials of degree $d$ with respect
to the reverse lexicographic ordering.

Let $S=k[X_1,\ldots,X_n]$. We view $S$ (the inverse system of $R$) as
an $R$-module, with the action of $R$ on $S$ being given by
$x_i\circ g=\frac{\partial}{\partial x_i}g$ for every $g\in S$ and
$i=1,\ldots,n$. If $A$ is a subset of $S$ we set
$$\ann(A)=\{f\in R:f\circ g=0\text{ for all }g\in A\}.$$ In the case
of positive characteristic, we have to replace the
\emph{differentiation action} of $S$ on $R$ defined above by the
\emph{contraction action}, given by
\[
  x_i\circ X_1^{a_1}X_2^{a_2}\cdots X_n^{a_n} =\begin{cases}
    X_1^{a_1}X_2^{a_2}\cdots X_i^{a_i-1}\cdots X_n^{a_n}&\text{ if
      $a_i \geq 1$}, \\ 0 &\text{otherwise}.\end{cases}
\]

Let $G$ be a generic form of degree $e$ in $S$. Then the Gorenstein
algebra $R/\ann(G)$ is compressed \cite{iarrobino1999power}, i.e., the
Hilbert series is equal to

$$\sum_{i=0}^{\lfloor e/2 \rfloor} \dim_k R_i t^i + \sum_{i=\lfloor e/2 \rfloor + 1}^{e} \dim_k R_{e-i} t^{i}.$$

Much is known about the generators of $\ann(G)$. If $e$ is even, then
$\ann(G)$ is generated in degree $e/2+1$ (this in fact holds for every
compressed Artinian algebra of even socle degree $e$ \cite[Example
4.7]{iarrobino1984compressed}). For $e$ odd, $\ann(G)$ is generated in
degree $\lfloor e/2 \rfloor +1$ if $n\ge 4$ or if $n=3$ and
$\lfloor e/2 \rfloor$ is odd (\cite[Corollary 4.5]{boij2023rate}). In
the case where $n=3$, $e$ is odd and $\lfloor e/2 \rfloor$ is even,
then $\ann(G)$ has generators in degree $\lfloor e/2 \rfloor +1$ and
at least one generator in degree $\lfloor e/2 \rfloor +2$
(\cite[Example 3.16]{BOIJ1999111}).

In our main result of this manuscript, we focus on determining the
initial ideal of $\ann(G)$ when $G\in S$ is a generic form of degree
$e$. For any monomial order on $R$ we show that $\inn(\ann(G))$ is
generated by the largest possible monomials in each degree. If we
consider the degree reverse lexicographic order, then this shows that
the ideal is not only weakly reverse lexicographic, but is in fact
reverse lexicographic, thus providing a connection to the currently
open Moreno-Soc\'ias' conjecture.

Let $H_e$ be the complete symmetric homogeneous form of degree $e$ in
$S$. The first author, Migliore, Mir\'o-Roig, and Nagel
\cite{boij2022waring} showed that $R/\ann(G)$ and $R/\ann(H_e)$ have
the same Hilbert series, so that in particular, $R/\ann(H_e)$ is also
compressed. This was achieved by determining the generators for
$\ann(H_e)$. They also showed that $R/\ann(H_e)$ has the strong
Lefschetz property by showing that the sum of the variables serves as
a strong Lefschetz element. Here, we extend this result by determining
a reduced Gr\"obner basis of $\ann(H_e)$ with respect to the degree
reverse lexicographical ordering. As a consequence, we obtain that
$\inn(\ann(H_e))=\inn(\ann(G))$, from which we get that the initial
ideal of $\ann(H_e)$ is reverse lexicographical and that $x_n$ serves
as a strong Lefschetz element.

Let $d=\lfloor e/2 \rfloor + 1$. The ideal $\ann(G)$ has at least
$\dim_k R_d - \dim_k R_{d-1}$ minimal generators of degree
$d$. Consider the ideal $I$ generated by at least as many general
forms of degree $d$, and possibly more of degrees $> d$. We provide a
connection between the initial ideal of $I$ and the initial ideal of
$\ann(G)$, with respect to the degree reverse lexicographical
ordering, by showing that they have the same total Betti numbers. This
relationship is unexpected since $R/I$ and $R/\ann(G)$ are not even
isomorphic as vector spaces.

The paper is organized as follows.  In Section \ref{sec:gen} we
determine the initial ideal of the annihilator ideal of a generic
form.  In Section \ref{sec:gb} we determine the Gr\"obner basis of the
annihilator ideal of the complete symmetric homogeneous form.  In
Section \ref{sec:betti} we determine homological invariants of a class
of strongly stable monomial ideals, of the initial ideal of the
annihilator ideal of a generic form, and of the initial ideal of the
ideal generated by sufficiently many generic forms.

\section{The generic case} \label{sec:gen}

\begin{theorem}\label{GenericIdealReverseLex}
  Let $G\in S$ be a generic homogeneous form of degree $e$ and let
  $\prec$ be any monomial order. Then for $d\in[e/2,e]$
  \[
    [\inn(\ann(G))]_d = \langle \text{first
      $\dim_k R_d-\dim_k R_{e-d}$ monomials of degree $d$}\rangle.
  \]
\end{theorem}

\begin{proof}
  We compute the degree $d$ part of the ideal $\ann (G)$ as the kernel
  of the catalecticant matrix $Cat^d_{e-d}(G)$ representing the
  pairing
  \[
    S_{d}\times S_{e-d} \longrightarrow k
  \]
  given by $(f,g)\mapsto (fg)\circ G$.  Let $\mathcal M_d$ be the
  monomial basis for $R_d$. We can write
  $G = \sum_{\mu\in\mathcal M_e} \frac{1}{\mu\circ \mu} a_\mu
  \mu$. Observe that $\mu\circ \mu = 1$ for all monomials $\mu$ if we
  are in characteristic $p$ using the contraction action. The
  catalecticant matrix for computing $\ann(G)$ in degree $d$ is given
  by
  \[
    \left( a_{\mu\nu}\right)_{\mu\in \mathcal M_{e-d},\nu\in \mathcal
      M_d}
  \]
  Let $S\subseteq \mathcal M_d$ and $T\subseteq \mathcal M_{e-d}$ be
  subsets of the same cardinality and consider the minor defined by
  them. If we order the rows and columns according to our monomial
  order, starting with the largest, we have that the indices of the
  entries of the catalecticant matrix is decreasing in rows and
  columns.  Hence, the product of the diagonal elements in the minor
  is a squarefree monomial in the coefficients $a_\mu$ that cannot
  occur in any other term of the minor. This shows that any minor is
  non-zero for a generic form. This means that we can row-reduce the
  catalecticant matrix to a form
  \[
    \begin{bmatrix}B & I \end{bmatrix}
  \]
  and the kernel is given by the columns of
  \[
    \begin{bmatrix}
      I\\-B
    \end{bmatrix}
  \]
  showing that the initial monomials of the basis of $[\ann(G)]_d$ are
  the first monomials in the monomial order that we have chosen.
\end{proof}

We now specify further properties of the ideal $\inn(\ann(G))$ when we
consider on $R$ the degree reverse lexicographic order induced by
$x_1 \succ \cdots \succ x_n$.

\begin{corollary}\label{LefschetzProperty}
  Let $G\in S$ ge a generic form of degree $e$. Consider on $R$ the
  degree reverse lexicographic order. Then $x_n$ is a strong Lefschetz
  element for $R/\inn(\ann(G))$.
\end{corollary}
\begin{proof}
  Let $J=\inn(\ann(G))$ and $B=R/J$. For $x_n$ to be a strong
  Lefschetz element for $B$ what is required is that
  $\cdot x_n^{e-2i}:B_i\to B_{e-i}$ is bijective for all $i<e/2$. But
  for such values of $i$, $B_i=R_i$ since $B$ is compressed. Also, by
  Theorem \ref{GenericIdealReverseLex}, a monomial generating set of
  $B_{e-i}$ consists of the smallest $\dim_k R_i$ monomials of degree
  $e-i$. These are precisely all monomials of degree $e-i$ that are
  divisible by $x_n^{e-2i}$. We conclude that
  $\cdot x_n^{e-2i}:B_i\to B_{e-i}$ is bijective for all $i<e/2$.
\end{proof}

For each $i\ge 0$ we will denote by $\mathcal A_i$ the monomial basis
of $k[x_1,\ldots,x_{n-1}]_i$.

\begin{corollary}\label{BasisAnnG}
  Let $G\in S$ ge a generic form of degree $e$. Consider on $R$ the
  degree reverse lexicographic order.

  If $e=2d+1$ is odd, then the minimal monomial basis of
  $\inn(\ann(G))$ in degree $d+1+i$ consists precisely of the
  monomials in $x^{2i}\mathcal A_{d+1-i}$ for each $0\le i\le d+1$.

  If $e=2d$ is even, then the minimal monomial basis of
  $\inn(\ann(G))$ in degree $d+1+i$ consists precisely of the
  monomials in $\mathcal A_{d+1}\cup x_n\mathcal A_{d}$ for $i=0$ and
  of the monomials in $x^{2i+1}\mathcal A_{d-i}$ for each
  $1\le i\le d$.
\end{corollary}

\begin{proof}
  $J=\inn(\ann(G))$ is generated in degrees
  $\lfloor e/2\rfloor+1= d+1\le j\le e+1$. By Theorem
  \ref{GenericIdealReverseLex}, $J$ is minimally generated in degree
  $d+1$ by the first $\dim_k R_{d+1}-\dim_kR_{e-d-1}$ monomials of
  degree $d+1$. In the case $e=2d+1$ these consist of the monomials in
  $\mathcal A_{d+1}$ and in the case $e=2d$ these consist of the
  monomials in $\mathcal A_{d+1}\cup x_n \mathcal A_{d}$.
 
  For each $j>d+1$, by Theorem \ref{GenericIdealReverseLex}, $J_{j-1}$
  is generated by all monomials of degree $j-1$ not divisible by
  $x_n^{2(j-1)-e}$, so that $R_1J_{j-1}$ is generated by all monomials
  of degree $j$ not divisible by $x_n^{2j-1-e}$. Therefore, the
  minimal generators of $J$ of degree $j$ are precisely the monomials
  of degree $j$ that are divisible by $x_n^{2j-1-e}$ and not by
  $x_n^{2j-e}$, i.e., the monomials in
  $x_n^{2j-1-e} \mathcal A_{e-j+1}$. Putting $j=d+1+i$ with $i\ge 1$
  we obtain the desired statement.
\end{proof}

\section{A Gr\"obner basis for the annihilator of the complete
  homogeneous polynomial} \label{sec:gb}

In this section we assume that the field $k$ has characteristic
zero. We also fix on $R=k[x_1,\ldots,x_n]$ the degree reserve
lexicographic order induced by $x_1 \succ \cdots \succ x_n$.

For each $e\ge 0$ denote by
$$H_e=\sum_{\underline {i} \in\mathbb N^n,|\underline i|=e}{X_1^{i_1}\cdots X_n^{i_n}} \in S$$
the \emph{complete symmetric polynomial} of degree $e$. In
\cite{boij2022waring} the authors have shown that $R/\ann(H_e)$ is
compressed (\cite[Theorem 2.11]{boij2022waring}). Thus, $R/\ann(H_e)$
and $R/\ann(G)$ have the same Hilbert function, where $G\in S$ is a
generic form of degree $e$. In their paper, the authors also determine
a set of generators of $\ann(H_e)$ (\cite[Theorem
2.12]{boij2022waring}) and show that $R/\ann(H_e)$ has the strong
Lefschetz Property.

In what follows, we focus on determining $\inn(\ann(H_e))$.  We will
show that $\inn(\ann(H_e))=\inn(\ann(G))$, from which it follows that
$\inn(\ann(H_e))$ is reverse lexicographical.  We achieve this by
specifying a Gröbner basis for $\inn(\ann(H_e))$. In order to present
such a basis, we introduce some notation.

\begin{definition}
  Let $\varphi:R\to R$ be the $k$-linear map given by
$$\varphi(x_1^{i_1}\cdots x_n^{i_n})=\frac 1{i_1!\cdots i_n!}x_1^{i_1}\cdots x_n^{i_n}$$ for each $(i_1,\ldots,i_n)\in\mathbb N^n$.
\end{definition}

\begin{definition}
  For each $\underline a=(a_1,\ldots,a_{n-1})\in\mathbb N^{n-1}$
  denote
$$F_{\underline a}=(x_1-x_n)^{a_1}\cdots (x_{n-1}-x_n)^{a_{n-1}}\in R.$$
\end{definition}

For every $d\ge 1$, in \cite[Theorem 2.12]{boij2022waring} the authors
have shown that $F_{\underline a}$, with $|\underline a|=d+1$,
annihilates $H_{2d}$ and $H_{2d+1}$.  We now show the following

\begin{proposition}\label{ElementsInJ}
  If $e=2d$ is even then
  $\{\varphi(x_n^{2i+1} F_{\underline a}):i=0,\ldots,d \text{ and }
  |\underline a|=d-i\}\subset \ann (H_{2d})$ and if $e=2d+1$ is odd
  then
  $\{\varphi(x_n^{2i} F_{\underline a}):i=0,\ldots,d+1 \text{ and }
  |\underline a|=d+1-i\}\subset \ann (H_{2d+1})$.
\end{proposition}
\begin{proof}
  We do induction on $e$. For $e=1$ the statement amounts to verifying
  that $x_i-x_n$, for $i=1,\ldots,n-1$, as well as $\frac 12 x_n^2$,
  annihilate $h_1=X_1+\cdots+X_n$, which is obvious.

  Suppose now that $e=2d$ is even and that the statement is true for
  all $e_0<2d$. Let $i\in\{0,\ldots,d\}$ and
  $\underline a=(a_1,\ldots,a_{n-1})$ be such that $|a|+i=d$. We need
  to show that $\varphi(x_n^{2i+1}F_{\underline a})\circ H_{2d}=0$.

  By \cite[Theorem 2.12]{boij2022waring}, considering $n+1$ variables
  $x_0,x_1,\ldots,x_n$, we have that
$$
f=\varphi'((x_1-x_n)^{a_1}\cdots
(x_{n-1}-x_n)^{a_{n-1}}(x_0-x_n)^{2i+1})=\varphi'(F_{\underline
  a}(x_0-x_n)^{2i+1}))
$$ annihilates $H'_{2(|a|+2i+1)-1}=H'_{2d+2i+1}$. Here $\varphi'$ stands for the map $\varphi$ defined on $k[x_0,x_1,\ldots,x_n]$ instead of $k[x_1,\ldots,x_n]$ and, for any $e'$, $H'_{e'}$ stands for the complete symmetric polynomial of degree $e'$ in the variables $X_0,X_1,\ldots,X_n$. 
We observe that
\begin{align*}
  f=\varphi'(F_{\underline a}(x_0-x_n)^{2i+1}) &=
                                                 \varphi'\left(\sum_{j=0}^{2i+1}{{2i+1}\choose{j}}(-1)^jx_n^jx_0^{2i+1-j}F_{\underline a} \right) \\
                                               &=\sum_{j=0}^{2i+1}{{2i+1}\choose{j}}\frac{(-1)^jx_0^{2i+1-j}}{(2i+1-j)!}\varphi(x_n^jF_{\underline a})
  \\
                                               &=\sum_{j=0}^{2i+1}b_jx_0^{2i+1-j}\varphi(x_n^jF_{\underline a}).
\end{align*}
Here the $b_j$ represent the suitable coefficients.  The third
equality holds because $x_0$ does not show up in $F_{\underline a}$.
Since
$$
H'_{2d+2i+1}=\sum_{t=0}^{2d+2i+1}X_0^t H_{2d+2i+1-t}
$$
and since
$$
(b_jx_0^{2i+1-j}\varphi(x_n^jF_{\underline a}))\circ (X_0^t
H_{2d+2i+1-t})= b_j'X_0^{t-(2i+1-j)}(\varphi(x_n^jF_{\underline
  a})\circ H_{2d+2i+1-t})
$$
for suitable coefficients $b_j'$, the coefficient of $X_0^{2i+1}$ in
$f\circ H'_{2d+2i+1}$ is
$$
c=\sum b'_{j}(\varphi(x_n^{j}F_{\underline a})\circ H_{2d+2i+1-t})
$$
where the sum ranges over all $j,t$ such that $0\le t\le 2d+2i+1$,
$0\le j\le 2i+1$ and $t-(2i+1-j)=2i+1$. Since we must have
$t\ge 2i+1$, we can write $t=2i+1+r$. Then $j=2i+1-r$ and we can
rewrite the coefficient of $X_0^{2i+1}$ in $f\circ H'_{2d+2i+1}$ as
$$
c=\sum_{r=0}^{2i+1} b'_{2i+1-r}(\varphi(x_n^{2i+1-r}F_{\underline
  a})\circ H_{2d-r})
$$
By the induction hypothesis, whenever $r\ge 1$ we have that
$\varphi(x_n^{2i+1-r}F_{\underline a})\circ H_{2d-r}=0$. Therefore
$$c=b'_{2i+1}\varphi(x_n^{2i+1}F_{\underline a})\circ H_{2d}.$$
Since $f$ annihilates $H'_{2d+2i+1}$, we must have $c=0$, i.e.,
$\varphi(x_n^{2i+1}F_{\underline a})\circ H_{2d}=0$, as we wanted to
show.

Suppose now that $e=2d+1$ is odd and that the statement is true for
all $e_0<2d+1$. In this case we need to show that
$\varphi(x_n^{2i}F_{\underline a})\circ H_{2d+1}=0$. for every
$i\in\{1,\ldots,d+1\}$ and $\underline a=(a_1,\ldots,a_{n-1})$ such
that $|a|+i=d+1$.  The proof follows along the same lines as in the
case of $e$ even, except this time one starts the argument from the
fact that
$$
f=\varphi'((x_1-x_n)^{a_1}\cdots
(x_{n-1}-x_n)^{a_{n-1}}(x_0-x_n)^{2i})=\varphi'(F_{\underline
  a}(x_0-x_n)^{2i}))
$$ annihilates $H'_{2(|a|+2i-1)}=H'_{2d+2i}$. We obtain, as before, that the coefficient of $X_0^{2i-1}$ in $f\circ H'_{2d+2i}$ is
$$
c=\sum_{r=0}^{2i} b'_{2i-r}(\varphi(x_n^{2i-r}F_{\underline a})\circ
H_{2d+1-r})
$$ for suitable nonzero coefficients $b'_j$.  
By the induction hypothesis, whenever $r\ge 1$ we have that
$\varphi(x_n^{2i-r}F_{\underline a})\circ H_{2d+1-r}=0$. Therefore
$$c=b'_{2i}\varphi(x_n^{2i}F_{\underline a})\circ H_{2d+1}.$$
Since $f$ annihilates $H'_{2d+2i}$, we must have $c=0$, i.e.,
$\varphi(x_n^{2i}F_{\underline a})\circ H_{2d+1}=0$.
\end{proof}

We can know specify a minimal Gröbner basis for $\ann(H_{e})$. We
recall that we are considering on $R$ the degree reverse lexicographic
ordering.

\begin{theorem}\label{GrobnerBasisResult}
  For any $e\ge 1$, if $G\in S$ is a generic form of degree $e$ then
$$\inn(\ann(G))=\inn(\ann(H_{e})).$$

Moreover, if $e=2d$ is even then
$$\mathcal G_{2d}=\{\varphi(F_{\underline a}):|\underline a|=d+1\}\cup\{\varphi(x_n^{2i+1} F_{\underline a}):i=0,\ldots,d \text{ and } |\underline a|=d-i\}$$ 
is a minimal Gröbner basis for $\ann(H_{2d})$.  If $e=2d+1$ is odd
then
$$\mathcal G_{2d+1}=\{\varphi(x_n^{2i} F_{\underline
  a}):i=0,\ldots,d+1 \text{ and } |\underline a|=d+1-i\}$$ is a
minimal Gröbner basis for $\ann(H_{2d+1})$.
\end{theorem}
\begin{proof}
  Suppose $e=2d+1$ is odd. By Proposition \ref{ElementsInJ} and
  \cite[Theorem 2.12]{boij2022waring}, $\mathcal G_{2d+1}$ is a subset
  of $\ann(H_{e})$. The initial terms of $\mathcal G_{2d+1}$ form the
  set
  $\{x_n^{2i}x_1^{a_1}\cdots x_{n-1}^{a_{n-1}}:i=0,\ldots,d+1 \text{
    and } a_1+\cdots +a_{n-1}=d+1-i \}$, which, by Corollary
  \ref{BasisAnnG}, is exactly the minimal monomial generating set of
  $\inn(\ann(H_{2d+1}))$. Since $R/\inn(\ann(H_{2d+1}))$ and
  $R/\inn(\ann(G))$ have the same Hilbert function, we conclude that
  $\inn(\ann(G))=\inn(\ann(H_{2d+1}))$ and that $\mathcal G_{2d+1}$ is
  a minimal Gröbner basis for $\ann(H_{2d+1})$.  For $e=2d$ even, the
  result follows in the same way.
\end{proof}

We finish this section with a couple of observations concerning the
reducedness of the Gröbner basis presented in Theorem
\ref{GrobnerBasisResult} and concerning the problem of determining a
minimal generating set for $\ann(H_e)$.

\begin{remark}
  It is easy to see that the Gröbner basis $\mathcal G_{2d+1}$ for
  $\ann(H_{2d+1})$ presented in Theorem \ref{GrobnerBasisResult} is
  reduced. However, the Gröbner basis $\mathcal G_{2d}$ for
  $\ann(H_{2d})$ is not in reduced form. This is because, given
  $(a_1,\ldots,a_{n-1})\in\mathbb N^{n-1}$ with
  $a_1+\cdots+a_{n-1}=d+1$, we have that
  \begin{align*}
    \varphi(F_{(a_1,\ldots,a_{n-1})})&=\varphi((x_1-x_n))^{a_1}\cdots (x_{n-1}-x_n)^{a_{n-1}}) \\
                                     &= \frac{x_1^{a_1}\cdots x_{n-1}^{a_{n-1}}}{a_1!\cdots a_{n-1}!}-x_n\left(\sum_{i=1}^{n-1}a_i\frac{x_1^{a_1}\cdots x_{i}^{a_{i}-1}\cdots x_{n-1}^{a_{n-1}}}{a_1!\cdots (a_i-1)!\cdots a_{n-1}!}\right)+\cdots
  \end{align*}
  has in its support monomials divisible by the leading monomials of
  the Gröbner basis elements
  $\varphi(x_nF_{(a_1,\ldots,a_i-1,\ldots,a_{n-1})})$ with
  $a_i\neq 0$. As such, replacing in $\mathcal G_{2d}$ each
  $\varphi(F_{(a_1,\ldots,a_{n-1})})$ with $a_1+\cdots+a_{n-1}=d+1$ by
$$\varphi(F_{(a_1,\ldots,a_{n-1})})+\sum_{i=1}^{n-1}a_i\varphi(x_nF_{(a_1,\ldots,a_i-1,\ldots,a_{n-1})})$$
we obtain a reduced Gröbner basis of $\ann(H_{2d})$.
\end{remark}

\begin{remark}
  According to \cite[Theorem 2.12]{boij2022waring}, for every
  $d\ge 0$, $\ann(H_{2d})$ is equigenerated in degree $d+1$, while
  $\ann(H_{2d+1})$ is generated in degrees $d+1$ and $d+2$. Therefore,
  a minimal set of generators of $\ann(H_{2d})$ is given by the
  elements of its minimal Gröbner basis $\mathcal G_{2d}$ that have
  degree $d+1$, i.e., by
  $\{\varphi(F_{\underline a}):|\underline a|=d+1\}\cup
  \{\varphi(x_nF_{\underline a}):|\underline a|=d\}.$ Similarly, the
  elements in $\{\varphi(F_{\underline a}):|\underline a|=d+1\}$ are
  minimal generators of $\ann(H_{2d+1})$. To determine a minimal set
  of generators of $\ann(H_{2d+1})$ it remains to check which elements
  in $\{\varphi(x_n^2 F_{\underline a}):|\underline a|=d\}$ are also
  minimal generators. We do not have an answer to this. However, our
  computational experiments point to the following conjecture.
\end{remark}

\begin{conjecture}
  If $d$ is odd, then $\ann(H_{2d+1})$ has no minimal generators in
  degree $d+2$, so that
  $\ann(H_{2d+1})=(\varphi(F_{\underline a}):|\underline a|=d+1)$.

  If $d$ is even, then $\ann(H_{2d+1})$ has exactly one minimal
  generator in degree $d+2$. In fact, we have
  $\ann(H_{2d+1})=(\varphi(F_{\underline a}):|\underline
  a|=d+1)+(\varphi(x_n^2(x_{n-1}-x_n)^{d}))$.
\end{conjecture}

\section{Total Betti numbers for a class of strongly stable
  ideals} \label{sec:betti}

Recall that a monomial ideal $J$ of $R$ is \emph{strongly stable} if
for every monomial $m\in J$, every variable $x_i$ that divides $m$ and
every $j<i$ the monomial $(x_j/x_i)m$ also belongs to $J$. In
particular, every weakly reverse lexicographic ideal is strongly
stable.

In this section we determine the total Betti numbers $\beta_p$ of
strongly stable monomial ideals generated in degrees $\geq d$ that
contain every monomial of degree $d$ not divisible by $x_n$. We then
show that $\beta_p$ equals the total Betti numbers of the initial
ideal of an ideal generated by generic forms in degrees $\geq d$ with
a sufficient number of generators of degree $d$, and likewise, that
$\beta_p$ equals the total Betti numbers of the initial ideal of the
annihilator of a generic form.

In the following statement, $\mathfrak m$ stands for
$(x_1,\ldots,x_n)$, the maximal homogeneous ideal of
$R=k[x_1,\ldots,x_n]$.

\begin{theorem}\label{SameBetti}
  Let $d\ge 2$. Let $J$ be a strongly stable monomial ideal of
  $R=k[x_1,\ldots,x_{n}]$ generated in degrees $\ge d$ and containing
  every monomial of degree $d$ not divisible by $x_n$. Suppose $R/J$
  is Artinian.  Then
$$\beta_p(R/J)=\beta_p(R/\mathfrak m^{d})=\sum_{i=1}^{n}{{{d+i-2}\choose {d-1}}{{i-1}\choose {p-1}}} \quad \text{for all }p\ge 1. $$
In particular, $J$ is minimally generated by $\dim_k R_d$ monomials.
\end{theorem}

\begin{proof}
  We start by proving the last statement, i.e., that $J$ is minimally
  generated by $\dim_k R_d$ monomials.  For each $i\ge 0$ let $r_i$
  denote the number of minimal generators of $J$ of degree $i$. We
  claim that $$\HF(R/J,i)-\HF(R/J,i+1)=r_{i+1}$$ for every $i\ge
  d$. Since
  \begin{align*}
    \HF(R/J,i+1)&=\dim_k R_{i+1}-\dim_kJ_{i+1}\\
                &=\dim_k R_{i+1}-\dim_k (R_1J_{i})-r_{i+1}\\
                &=\dim_k (R_{i+1}/R_1J_{i})-r_{i+1},
  \end{align*}
  the claim is equivalent to showing that
  $$\dim_k (R_{i}/J_{i})=\dim_k (R_{i+1}/R_1J_{i}).$$

  For this, we observe that multiplication by $x_n$ induces a
  bijection between the monomial basis of $R_{i}/J_{i}$ and the
  monomial basis of $R_{i+1}/R_1J_{i}$. The surjectiveness comes from
  the fact that $J_i$ contains every monomial of degree $i$ not
  divisible by $x_n$, thus forcing every element of $R_{i+1}/R_1J_{i}$
  to be a multiple of $x_n$. To show injectiveness, suppose that $m$
  is a monomial of degree $i$ such that $x_n m = x_t m'$ where
  $1\le t\le n$ and $m'$ is a monomial in $J_i$. If $t=n$ then
  $m=m'\in J_i$ and we are done. If $t<n$ then $m'$ must be divisible
  by $x_n$ and we get $m=m'(x_t/x_n)$, which, since $J$ is strongly
  stable, belongs to $J_i$. This proves the claim.

  Let $s$ be the socle degree of $R/J$. Since $J$ is generated in
  degrees greater or equal to $d$, we have
  \begin{align*}
    \dim_k R_d-r_d &=\HF(R/J,d) \\
                   &= \sum_{i=d}^{s} \HF(R/J,i)-\HF(R/J,i+1) \\
                   &= \sum_{i=d}^{s} r_{i+1}. 
  \end{align*}
  We conclude that $r_d+\cdots+r_{s+1}=\dim_k R_d$, i.e., $J$ is
  minimally generated by $\dim_k R_d$ monomials.

  Consequently, $J$ is minimally generated by all
  $\dim_k R_{d}-\dim_k R_{d-1}$ monomials of degree $d$ that are not
  divisible by $x_n$ and by $\dim_k R_{d-1}$ monomials divisible by
  $x_n$.  More precisely, for each $1\le i<n$, $J$ contains exactly
  $a_i=\dim_k k[x_1,\ldots,x_{i}]_{d-1}={{d+i-2}\choose {d-1}}$
  minimal generators that are divisible by $x_i$ and not by any $x_j$
  with $j>i$. Also $J$ contains exactly
  $a_n=\dim_k R_{d-1}={{d+n-2}\choose {d-1}}$ minimal generators
  divisible by $x_n$.

  As $J$ is strongly stable, the minimal free resolution of $R/J$ is
  given by the Eliahou-Kervaire resolution. According to
  \cite[Corollary 28.12]{peeva2010graded}, the total Betti numbers
  $\beta_p(R/J)$ are given by the formula
$$\beta_p(R/J)=\sum_{i=1}^{n}{a_i{{i-1}\choose {p-1}}}=\sum_{i=1}^{n}{{{d+i-2}\choose {d-1}}{{i-1}\choose {p-1}}}$$
for all $p\ge 1$. Since $\mathfrak m^d$ satisfies the hypothesis of
the theorem, the total Betti numbers of $R/\mathfrak m^d$ also agree
with this formula.
\end{proof}

As a corollary, we obtain that if $I$ is an ideal of $R$ generated by
generic forms of degrees $\ge d$, of which a sufficient number of them
have degree exactly $d$, then adding more generic forms of degrees
$\ge d$ to $I$ will not change the total Betti numbers of
$R/\inn(I)$. In particular, it will not change the number of minimal
generators of $\inn(I)$.

\begin{corollary}\label{NumberGens}
  Let $d\ge 2$ and $n\ge 3$. Let $I$ be an ideal of
  $R=k[x_1,\ldots,x_{n}]$ generated by generic forms of degrees
  greater than or equal to $d$, of which at least
  $\dim_k R_d-\dim_k R_{d-1}$ have degree $d$. Then
$$\beta_p(R/\inn(I))=\beta_p(R/\mathfrak m^{d})=\sum_{i=1}^{n}{{{d+i-2}\choose {d-1}}{{i-1}\choose {p-1}}} \quad \text{for all }p\ge 1. $$
In particular, $\inn(I)$ is minimally generated by $\dim_k R_d$
monomials.
\end{corollary}

\begin{proof}
  Since $I$ is generated by generic forms, $\inn(I)$ equals the
  generic initial ideal of $I$, hence $\inn(I)$ is a strongly stable
  ideal by \cite[Theorem 28.4]{peeva2010graded}.  Also, as at least
  $\dim_k R_d-\dim_k R_{d-1}$ of the generators of $I$ have degree
  $d$, $\inn(I)$ contains all the monomials of degree $d$ that are not
  divisible by $x_n$. Since $d\ge 2$ and $n\ge 3$, we have
  $\dim_k R_d-\dim_k R_{d-1}\ge n$, thus $R/I$ is Artinian. The
  statement now follows from Theorem \ref{SameBetti}.
\end{proof}

\begin{example}
  Let $I$ be an ideal in $R=k[x_1,x_2,x_3]$ generated by
  $\dim_{k} R_{11} - \dim_{k} R_{10}=12$ generic forms of degree
  $11$. Let $I'$ be $I$ plus the ideal generated by a generic form of
  degree $12$, and let $I''$ be $I$ plus the ideal generated by the
  non-generic form $(x_1+x_2+x_3)^5 x_3^7$. All three ideals will have
  socle in degree $13$, and the non-trivial parts of the Betti tables
  equals

$$
\begin{matrix}
  & 1 & 2 & 3\\
  \text{total:} & 78 & 143 & 66\\
  10: & 12 & 11 & .\\
  11: & 11 & 22 & 11\\
  12: & 22 & 44 & 22\\
  13: & 33 & 66 & 33
\end{matrix} \hspace{1.5cm}
\begin{matrix}
  & 1 & 2 & 3\\
  \text{total:} & 78 & 143 & 66\\
  10: & 12 & 11 & .\\
  11: & 12 & 24 & 12\\
  12: & 24 & 48 & 24\\
  13: & 30 & 60 & 30
\end{matrix} \hspace{1.5cm}
\begin{matrix}
  & 1 & 2 & 3\\
  \text{total:} & 76 & 140 & 65\\
  10:  & 12 & 11 & .\\
  11:  & 11 & 22 & 11\\
  12:  & 22 & 47 & 24\\
  13: & 30 & 60 & 30
\end{matrix}$$ for $\inn(I)$, $\inn(I')$, and $\inn(I'')$,
respectively. In particular, the total Betti numbers for $\inn(I)$ and
$\inn(I')$ agree, while those for $\inn(I'')$ differ due to the fact
that $I''$ is not generated solely by generic forms.
\end{example}

\begin{remark}\label{remarkGradedBetti}
  Let $I=\ann(G)$, where $G\in S$ is a generic form of degree
  $e\ge 1$. According to Theorem \ref{GenericIdealReverseLex},
  $\inn(I)$ is a reverse lexicographic, hence strongly stable, ideal
  that is generated in degrees $\lfloor e/2\rfloor+1= d+1\le q\le e+1$
  and contains all monomials of degree $d+1$ not divisible by
  $x_n$. Hence, $\inn(I)$ satisfies the requirements of Theorem
  \ref{SameBetti}, so that
  $\beta_p(R/\inn(I))=\beta_p(R/\mathfrak m^{d+1})$ for all $p\ge 0$.

  However, by Corollary \ref{BasisAnnG}, we know exactly the basis
  elements of $\inn(I)$, so we can actually specify all the graded
  Betti numbers of $R/\inn(I)$.

  Since $\inn(I)$ is a strongly stable ideal, by \cite[Corollary
  28.12]{peeva2010graded} we have that
$$\beta_{p,p+q}(R/\inn(I))=\sum_{i=1}^{n}{a_{i,q}}{{i-1}\choose {p-1}}$$
for all $p\ge 1$ and $q\ge 0$, where $a_{i,q}$ denotes the number of
minimal generators of $\inn(I)$ that have degree $q$ and that are
divisible by $x_i$ but not by any $x_j$ with $j>i$.

In degree $d+1$ we have that $\inn(I)$ contains all monomials not
divisible by $x_n$, so that
$a_{i,d+1}=\dim_k k[x_1,\ldots,x_{i}]_{d}={{d+i-1}\choose {d}}$ for
all $i<n$ and $a_{i,q}=0$ for all $i<n$, $q>d+1$. From Corollary
\ref{BasisAnnG} we get
$a_{n,q}=\dim_k k[x_1,\ldots,x_{n-1}]_{e+1-q}={{e-q+n-1}\choose
  {n-2}}$ for all $q>d+1$. We also see that $a_{n,d+1}=0$ if $e=2d+1$
is odd and that
$a_{n,d+1}=\dim_k k[x_1,\ldots,x_{n-1}]_{d}={{d+n-2}\choose {d}}$ if
$e=2d$ is even.

We conclude that
$$\beta_{p,p+q}(R/\inn(I))={{{e-q+n-1}\choose {n-2}}{{n-1}\choose {p-1}}}$$
for all $p\ge 1$ and $q>d+1$, whereas
$$\beta_{p,p+d+1}(R/\inn(I))=\left\{
\begin{array}{ll}
  \sum_{i=1}^{n-1}{{{d+i-1}\choose {d}}{{i-1}\choose {p-1}}}   &  \text{if } e=2d+1 \text{ is odd}\\
  {{d+n-2}\choose {d}}+\sum_{i=1}^{n-1}{{{d+i-1}\choose {d}}{{i-1}\choose {p-1}}}   &  \text{if } e=2d \text{ is even}
\end{array}
\right.$$ for all $p\ge 1$.
\end{remark}

Our computational experiments also give evidence for the following.

\begin{conjecture}\label{InitialInclusion} 
  Let $G\in S$ is a generic form of degree $e\ge 1$. Let $I$ be an
  ideal of $R$ generated by generic forms of degrees greater than or
  equal to $d=\lfloor e/2\rfloor+1$, of which at least
  $\dim_k R_d-\dim_k R_{d-1}$ have degree $d$. Then
$$\inn(\ann(G))\subseteq \inn(I).$$
\end{conjecture}

If $I$ is as in Conjecture \ref{InitialInclusion} then, by Corollary
\ref{NumberGens}, we know that the minimal number of generators of
$\inn(I)$ is $\dim_k R_d$. However, we do not know how many minimal
generators $\inn(I)$ has in each degree. If that could be determined,
then, as in Remark \ref{remarkGradedBetti}, we would know the whole
Betti table of $R/\inn(I)$ and, as a consequence, one could possibly
determine the Hilbert function of $R/I$ and check if it matches with
the Hilbert function predicted by Fr\"oberg's conjecture.

\section*{Acknowledgements}
Experiments in Macaulay2 \cite{M2} have been crucial to derive the
results presented in the paper. The first author was partially funded
by the Swedish Research Council VR2024-04853, the second author was
supported by the Sverker Lerheden Foundation, and third author was
partially supported by the Swedish Research Council grant
VR2022-04009.
\appendix


\begin{thebibliography}{10}

\bibitem{BOIJ1999111}
Mats Boij.
\newblock Betti numbers of compressed level algebras.
\newblock {\em Journal of Pure and Applied Algebra}, 134(2):111--131, 1999.

\bibitem{boij2023rate}
Mats Boij, Emanuela De~Negri, Alessandro De~Stefani, and Maria~Evelina Rossi.
\newblock On the rate of generic gorenstein $ k $-algebras.
\newblock {\em arXiv preprint arXiv:2304.14842}, 2023.

\bibitem{boij2022waring}
Mats Boij, Juan Migliore, Rosa~M. Mir{\'o}-Roig, and Uwe Nagel.
\newblock Waring and cactus ranks and strong lefschetz property for
  annihilators of symmetric forms.
\newblock {\em Algebra \& Number Theory}, 16(1):155--178, 2022.

\bibitem{Froberg}
Ralf Fr{\"o}berg.
\newblock An inequality for {Hilbert} series of graded algebras.
\newblock {\em Mathematica Scandinavica}, 56:117--144, 1985.

\bibitem{M2}
Daniel~R. Grayson and Michael~E. Stillman.
\newblock Macaulay2, a software system for research in algebraic geometry.
\newblock Available at \url{http://www2.macaulay2.com}.

\bibitem{iarrobino1984compressed}
Anthony Iarrobino.
\newblock Compressed algebras: Artin algebras having given socle degrees and
  maximal length.
\newblock {\em Transactions of the American Mathematical Society},
  285(1):337--378, 1984.

\bibitem{iarrobino1999power}
Anthony Iarrobino and Vassil Kanev.
\newblock {\em Power sums, Gorenstein algebras, and determinantal loci}.
\newblock Springer Science \& Business Media, 1999.

\bibitem{MORENOSOCIAS2003263}
Guillermo Moreno-Soc\'ias.
\newblock Degrevlex gröbner bases of generic complete intersections.
\newblock {\em Journal of Pure and Applied Algebra}, 180(3):263--283, 2003.

\bibitem{PARDUE2010579}
Keith Pardue.
\newblock Generic sequences of polynomials.
\newblock {\em Journal of Algebra}, 324(4):579--590, 2010.

\bibitem{peeva2010graded}
Irena Peeva.
\newblock {\em Graded syzygies}, volume~14.
\newblock Springer Science \& Business Media, 2010.

\end{thebibliography}
\end{document}